\documentclass[12pt]{amsart}

\usepackage{array,longtable}
\usepackage[utf8]{inputenc}
\usepackage{bbm}
\usepackage{amssymb,latexsym,fancyhdr,color,graphics}
\usepackage{amsmath, amsthm}
\usepackage{setspace}
\usepackage[dvips]{graphicx}
\usepackage[pdftex]{hyperref}
\usepackage[parfill]{parskip}
\usepackage{enumerate}

\usepackage{multirow}
\usepackage{booktabs}
\usepackage{caption}

\newtheorem{defi}{\bf Definition}

\newtheorem{lem}{\bf Lemma}
\newtheorem*{asm}{\bf Assumption}
\newtheorem{prop}{\bf Proposition}
\newtheorem{thm}{\bf Theorem}
\newtheorem{coro}{\bf Corollary}
\newtheorem{ex}{\bf Example}
\newtheorem{rem}{\bf Remark}

\newcommand{\eE}{\mathbf{E}}
\newcommand{\pP}{\mathbf{P}}

\newcommand{\nN}{\mathbb{N}}
\newcommand{\rR}{\mathbb{R}}
\newcommand{\qQ}{\mathbb{Q}}

\newcommand{\cadlag}{c\`{a}dl\`{a}g }

\title[Weak Limits of Autoregressive Processes]{Weak Limits of Random Coefficient Autoregressive Processes and their Application in Ruin Theory}

\begin{document}

\maketitle

\begin{center}
{\large  Y. DONG} \\
\smallskip
LAREMA, UMR 6093, UNIV Angers, CNRS, SFR MathSTIC, France \\
\texttt{ycdong@fudan.edu.cn}

\vspace{0.5cm}

{\large  J. SPIELMANN} \\
\smallskip
LAREMA, UMR 6093, UNIV Angers, CNRS, SFR MathSTIC, France \\
\texttt{jerome.spielmann@univ-angers.fr (corresponding)}
\end{center} 

\vspace{0.2cm}

\begin{abstract}
We prove that a large class of discrete-time insurance surplus processes converge weakly to a generalized Ornstein-Uhlenbeck process, under a suitable re-normalization and when the time-step goes to $0$. Motivated by ruin theory, we use this result to obtain approximations for the moments, the ultimate ruin probability and the discounted penalty function of the discrete-time process.
\end{abstract}



\section{Introduction}

\footnotetext{This paper has been accepted for publication in {\it Insurance: Mathematics and Economics} and will be published in March 2020. It is already available online with the DOI : \url{https://doi.org/10.1016/j.insmatheco.2019.12.001}. (CC BY-NC-ND 2020)}

Let $(\xi_k)_{k \in \nN^*}$ and $(\rho_k)_{k \in \nN^*}$ be two i.i.d. and independent sequences of random variables, with $\rho_k > 0$ $(\pP-a.s.)$ for all $k \in \nN^*$. The \textit{autoregressive process of order 1 with random coefficients}, abbreviated RCA(1) or RCAR(1), see e.g. \cite{RCAR}, is given by 
\begin{equation}\label{eq_RCAR}
\theta_k = \xi_k + \theta_{k -1}\rho_k, k \in \nN^*.
\end{equation}
and $\theta_0 = y \in \rR$. Such processes, which are also called \textit{stochastic recurrence or difference equations}, appear frequently in applied probability. For example, it is suggested in \cite{andel1976} that RCA processes could be useful in problems related to hydrology, meteorology and biology. We also refer to \cite{vervaatappli} for a more exhaustive list of examples. In ruin theory, the RCA(1) process is a classic model for the surplus capital of an insurance company where $(\xi_k)_{k \in \nN^*}$ represents a stream of random payments or income and $(\rho_k)_{k \in \nN^*}$ represents the random rates of return from one period to the next, see for example \cite{nyrhinen1999}, \cite{nyrhinen2001}, \cite{nyrhinen2012} and \cite{tang2003}.

In this paper, we prove the convergence of the process (\ref{eq_RCAR}) when the time-step goes to $0$ and under a suitable re-normalization to the generalized Ornstein-Uhlenbeck (GOU) process given by
\begin{equation}\label{eq_GOU}
Y_t = e^{R_t}\left(y + \int_{0+}^t e^{-R_{s-}} dX_s\right), t \geq 0,
\end{equation}
where $R = (R_t)_{t \geq 0}$ and $X = (X_t)_{t \geq 0}$ are independent stable L\'{e}vy processes with drift. One of the main uses of weak convergence is to prove the convergence of certain functionals of the path of the processes to the functional of the limiting process and to use the value of the latter as an approximation for the former, when the steps between two payments and their absolute values are small. Motivated by ruin theory, we will use this technique to prove the convergence of the ultimate ruin probability, a simple form of the discounted penalty function and the moments.

In general, (\ref{eq_GOU}) is chosen as a model for insurance surplus processes with investment risk on an \textit{a priori} basis. The ruin problem is then studied under the heading "ruin problem with investment" for different choices of $R$ and $X$. We refer to \cite{paulsen2008} and the references therein for an overview of the relevant literature. The main convergence results of this paper could thus also be seen as a theoretical justification for the continuous-time model (\ref{eq_GOU}) in the context of models for insurance surplus processes with both insurance and market risks, in the same spirit as the results in \cite{duffieprotter1992}.

In actuarial mathematics, similar convergence results and approximations of functionals of surplus processes are a well-developed line of research. In \cite{iglehart1969} it is shown that the compound Poisson process with drift converges weakly to a Brownian motion with drift and it is shown that the finite-time and ultimate ruin probability converge to those of the limiting model. These results are extended to more general jump times in \cite{grandell1977} and to more general jump sizes in \cite{burn2000} and \cite{furrer1997}. Similar convergence results are proven for the integral of a deterministic function w.r.t. a compound Poisson process in \cite{harrison1977}, this corresponds to the assumption that the insurance company can invest at a deterministic interest rate. Some of the previous results are generalized in \cite{paulsen1997b}, where it is shown that a general model with a jump-diffusion surplus process and stochastic jump-diffusion investment converges to a particular diffusion process.

More closely related to our result are the papers \cite{cumberland1982} and \cite{dufresne1989}. In \cite{cumberland1982}, it is shown that the AR(1) process (i.e. when the coefficients $\rho_k$ are deterministic and constant) converges weakly to a standard Ornstein-Uhlenbeck process. In \cite{dufresne1989}, it is shown that when the variables $\xi_k$ are deterministic and satisfy some regularity conditions, we have a similar weak convergence result where the process $X$ in (\ref{eq_GOU}) is replaced by a deterministic function.

The results in \cite{duffieprotter1992} are also closely related. In that paper, the authors study the weak convergence of certain discrete-time models to continuous-time models appearing in mathematical finance and prove the convergence of the values of certain functionals such as the call option price. In particular, for the case $\xi_k = 0$, for all $k \in \nN^*$, they show, using the same re-normalization as we do (see below at the beginning of Section 2), that the discrete-time process (\ref{eq_RCAR})  converges to the Dol\'{e}ans-Dade exponential of a Brownian motion with drift. This generalizes the famous paper \cite{cox1979} where it is shown that the exponential of a simple random walk correctly re-normalized converges to the Black-Scholes model.

Finally, the relationship between the discrete-time process (\ref{eq_RCAR}) and (\ref{eq_GOU}) was also studied in \cite{dehaan1989}, where it is shown that GOU processes are continuous-time analogues of RCA(1) processes in some sense. More precisely, they show that any continuous-time process $S = (S_t)_{t \geq 0}$ for which the sequence $(S_{nh})_{n \in \nN^*}$ of the process sampled at rate $h > 0$ satisfies an equation of the form (\ref{eq_RCAR}), for all $h > 0$, with some additional conditions, is a GOU process of the form (\ref{eq_GOU}), where $X$ and $R$ are general L\'{e}vy processes. Our main result is coherent with this analogy but does not seem to be otherwise related.

The rest of the paper is structured as follows: after introducing the assumptions and notations, we prove the weak convergence of $(\ref{eq_RCAR})$ to $(\ref{eq_GOU})$ in Theorem 1. From this result, we deduce the convergence in distribution of the ruin times in Theorem \ref{cor_convRuinProb}. Then, we give sufficient conditions for the convergence of a simple form of the discounted penalty function in Theorem 3, of the ultimate ruin probability in Theorem 4 and of the moments in Theorem 5, when $\xi_1$ and $\ln(\rho_1)$ are both square-integrable. We illustrate these results using examples from actuarial theory and mathematical finance.

\section{Weak Limits of Autoregressive Processes and Convergence of the Ruin Times}

In this section, we show that the discrete-time process converges weakly to the GOU process and prove the convergence in distribution of the ruin times.

\subsection{Assumptions and Convergence Results}

We will use the following set of assumptions.

\begin{asm}[$\mathbf{H^{\alpha}}$]\rm
We say that a random variable $Z$ satisfies $(\mathbf{H^{\alpha}})$ if its distribution function satisfies
$$\pP(Z \leq -x) \sim k^Z_1 x^{-\alpha} \text{ and } \pP(Z \geq x) \sim k^Z_2 x^{-\alpha},$$
as $x \to \infty$, for some $1 < \alpha < 2$, where $k^Z_1, k^Z_2$ are constants such that $k^Z_1 + k^Z_2 > 0$. Note that this implies that $\eE(|Z|) < \infty$.
\end{asm}

\begin{asm}[$\mathbf{H^2}$]\rm
We say that a random variable $Z$ satisfies $(\mathbf{H^2})$ if $Z$ is square-integrable with $\mathrm{Var}(Z) > 0$, where $\mathrm{Var}(Z)$ is the variance of $Z$.
\end{asm}

We now introduce some notations and recall some classical facts about weak convergence on metric spaces, stable random variables and L\'{e}vy processes.

Recall that the space $D$ of c\`{a}dl\`{a}g functions $\rR_+ \to \rR$ can be equipped with the Skorokhod metric which makes it a complete and separable metric space, see e.g. Section VI.1, p.324 in \cite{js2003}. Let $\mathcal{D}$ be the Borel sigma-field for this topology. Given a sequence of random elements $Z^{(n)} : (\Omega^{(n)}, \mathcal{F}^{(n)}, \pP^{(n)}) \mapsto (D, \mathcal{D})$, with $n \geq 1$, we say that $(Z^{(n)})_{n \geq 1}$ converges weakly or in distribution to $Z : (\Omega, \mathcal{F}, \pP) \mapsto (D, \mathcal{D})$, if the laws of $Z^{(n)}$ converge weakly to the law of $Z$, when $n \to \infty$. We denote weak convergence by $Z^{(n)} \overset{d}{\to} Z$ and we use the same notation for the weak convergence of measures on $\rR$. We refer to Chapter VI, p.324 in \cite{js2003} for more information about these notions.

Concerning stable random variables $Z$ of index $\alpha$, the most common way to define them is trough their characteristic functions:
$$\eE(e^{iuZ}) = \exp[i\gamma u - c|u|^{\alpha}(1 - i\beta \mathrm{sign}(u)z(u, \alpha))],$$
where $\gamma \in \rR$, $c > 0$, $\alpha \in (0, 2]$, $\beta \in [-1, 1]$ and
$$z(u, \alpha) = 
\begin{cases}
\tan\left(\frac{\pi \alpha}{2}\right) \text{ if } \alpha \neq 1,\\
-\frac{2}{\pi}\ln|u| \text{ if } \alpha = 1.
\end{cases}
$$
Stable L\'{e}vy processes $(L_t)_{t \geq 0}$ are L\'{e}vy processes such that $L_t$ is equal in law to some stable random variable, for each $t \geq 0$, with fixed parameters $\beta \in [-1, 1]$ and $\gamma = 0$ (see e.g. Definition 2.4.7 p.93 in \cite{embrechts1997}.)

Finally, note that if $(Z_k)_{k \in \nN^*}$ is a sequence of i.i.d. random variables such that $Z_1$ satisfies either $(\mathbf{H^{\alpha}})$ or $(\mathbf{H^2})$, then there exists a stable random variable $K_{\alpha}$ and a constant $c_{\alpha} > 0$ such that
\begin{equation}\label{eq_convsum}
\sum_{k = 1}^n\frac{Z_k - \mu_Z}{c_{\alpha} n^{1/\alpha}} \overset{d}{\to} K_{\alpha},
\end{equation}
as $n \to \infty$, where $\mu_Z = \eE(Z_1)$. In fact, when $Z_1$ satisfies $(\mathbf{H^2})$, $\alpha = 2$, $c_{\alpha} = 1$ and $K_{\alpha}$ is the standard normal distribution with variance $\mathrm{Var}(Z_1)$. (See e.g. Section 2.2 p.70-81 in \cite{embrechts1997} for these facts.)

\begin{rem}\rm
The assumptions $(\mathbf{H^{\alpha}})$ and $(\mathbf{H^2})$ do not cover all possible cases. For example, the random variable $Z$ whose distribution function satisfies $\pP(Z \leq -x) \sim x^{-2}$, as $x \to \infty$, satisfies neither $(\mathbf{H^{\alpha}})$ nor $(\mathbf{H^2})$. In that case, the proofs of Theorems 1 and 2 below can still work. Then, the normalizing sequence $(c_{\alpha} n^{1/\alpha})_{n \in \nN^*}$ in (\ref{eq_convsum}) is replaced by $(S_{\alpha}(n) n^{1/\alpha})_{n \in \nN^*}$, where $S_{\alpha} : \rR_+^* \to \rR_+^*$ is a slowly varying function (see Theorem 2.2.15 in \cite{embrechts1997}) and the definitions of the sequences $(\xi^{(n)}_k)_{k \in \nN^*}$ and $(\rho^{(n)}_k)_{k \in \nN^*}$ below have to be adapted by replacing $c_{\alpha}$ by $S_{\alpha}(n)$. Moreover, to be able to obtain (\ref{eq_thm1_convsn}) in the proof of Theorem 1 below, we need an additional assumption ; for example, it is enough that $\lim_{x \to \infty} S_{\alpha}(x) > 0$ exists and is finite.
\end{rem}

We now turn to the presentation of the main assumptions and results of this section. 

\begin{asm}[$\mathbf{H}$]\rm
We assume that $(\xi_k)_{k \in \nN^*}$ and $(\rho_k)_{k \in \nN^*}$ are two i.i.d. and independent sequences of random variables, with $\rho_k > 0$ $(\pP-a.s.)$ for all $k \in \nN^*$, and such that $\xi_1$ (resp. $\ln(\rho_1)$) satisfies either $(\mathbf{H^{\alpha}})$ or $(\mathbf{H^2})$ (resp. $(\mathbf{H^{\beta}})$ or $(\mathbf{H^2})$.) We denote by $c_{\alpha}$ (resp. $c_{\beta}$) the constant and by $K_{\alpha}$ (resp. $K_{\beta}$) the limiting stable random variable appearing in $(\ref{eq_convsum})$. Denote by $(L^{\alpha}_t)_{t \geq 0}$ (resp. $(L^{\beta}_t)_{t \geq 0}$) the stable L\'{e}vy processes obtained by putting $L^{\alpha}_1 \overset{d}{=} K_{\alpha}$ (resp. $L^{\beta}_1 \overset{d}{=} K_{\beta}$).
\end{asm}

Fix $n \in \nN^*$, we want to divide the time interval into $n$ subintervals of length $1/n$ and update the discrete-time process at each time point of the subdivision. To formalize this, we define the following process
\begin{equation}\label{eq_AR}
\theta^{(n)}\left(\frac{k}{n}\right) = \xi^{(n)}_k + \theta^{(n)}\left(\frac{k-1}{n}\right)\rho^{(n)}_k, k \in \nN^*,
\end{equation}
where $(\xi^{(n)}_k)_{k \in \nN^*}$ and $(\rho^{(n)}_k)_{k \in \nN^*}$ have to be defined from the initial sequences. Following an idea in \cite{dufresne1989}, we let $\mu_{\xi} = \eE(\xi_1)$ and $\mu_{\rho} = \eE(\ln(\rho_1))$ and define:
$$\xi^{(n)}_k = \frac{\mu_{\xi}}{n} + \frac{\xi_k - \mu_{\xi}}{c_{\alpha}n^{1/\alpha}}$$
and $\rho^{(n)}_k = \exp(\gamma^{(n)}_k)$ where
$$\gamma^{(n)}_k = \frac{\mu_{\rho}}{n} + \frac{\ln(\rho_k) - \mu_{\rho}}{c_{\beta}n^{1/\beta}}.$$
These definitions ensure that 
$$\eE\left(\sum_{k = 1}^{n}\xi_k^{(n)}\right) = \mu_{\xi} \hspace{2mm} \text{ and } \hspace{2mm} \eE\left(\sum_{k = 1}^{n}\ln(\rho_k^{(n)})\right) = \mu_{\rho}.$$
Moreover, when $\xi_1$ and $\ln(\rho_1)$ both satisfy $(\mathbf{H^2})$, we choose $\alpha = \beta = 2$ and $c_{\alpha} = c_{\beta} = 1$, and then we have the following variance stabilizing property:
$$\mathrm{Var}\left(\sum_{k = 1}^{n}\xi_k^{(n)}\right) = \mathrm{Var}(\xi_1)\hspace{2mm} \text{ and } \hspace{2mm}  \mathrm{Var}\left(\sum_{k = 1}^{n}\ln(\rho_k^{(n)})\right) = \mathrm{Var}(\ln(\rho_1)).$$

Finally, we define the filtrations $\mathcal{F}^{(n)}_0 = \{\emptyset, \Omega\}$, $\mathcal{F}^{(n)}_k = \sigma((\xi^{(n)}_i, \rho^{(n)}_i), i = 1, \dots, k)$, $k \in \nN^*$ and $\mathcal{F}^{(n)}_t = \mathcal{F}^{(n)}_{[nt]}$, for $t \geq 0$, where $[.]$ is the floor function and define $\theta^{(n)}$ as the (continuous-time) stochastic process given by 
$$\theta^{(n)}_t = \theta^{(n)}\left(\frac{[nt]}{n}\right), t \geq 0.$$

\begin{thm}\label{thm_1}
Under $\mathbf{(H)}$, we have $\theta^{(n)} \overset{d}{\to} Y$, as $n \to \infty$, where $Y = (Y_t)_{t \geq 0}$ is the GOU process $(\ref{eq_GOU})$ with $X_t = \mu_{\xi} t + L^{\alpha}_t$ and $R_t = \mu_{\rho} t + L^{\beta}_t$, for all $t \geq 0$. In addition, $Y$ satisfies the following stochastic differential equation :
\begin{equation}\label{eq_SDE}
Y_t = y + X_t + \int_{0+}^tY_{s-} d\hat{R}_s, t \geq 0,
\end{equation}
where
$$\hat{R}_t = R_t + \frac{1}{2}\langle R^c \rangle_t + \sum_{0 < s \leq t}\left(e^{\Delta R_s} - 1 - \Delta R_s\right), t \geq 0,$$
and $R^c$ is the continuous martingale part of $R$ and $\Delta R_t$ is its jump at time $t \geq 0$.
\end{thm}

\begin{ex}[Pareto losses and stable log-returns]\rm
The assumption $\mathbf{(H^{\alpha})}$ is quite general and simple to check. To illustrate it we take the negative of a Pareto (type I) distribution with shape parameter $1 < \alpha < 2$ for the loss $\xi_1$, i.e. the random variable defined by its distribution function $F_{\xi}(x) = (-x)^{-\alpha}$, for $x \leq -1$. The condition on $\alpha$ ensures that $\xi_1$ has a finite first moment, but an infinite second moment. Moreover, $\xi_1$ then satisfies $\mathbf{(H^{\alpha}})$, with constants $k_1^{\xi} = 1$ and $k_2^{\xi} = 0$. We also have that $\mu_{\xi} = -\alpha/(\alpha - 1)$ and that
$$\sum_{k = 1}^n\frac{\xi_k - \mu_{\xi}}{c_{\alpha, \xi} n^{1/\alpha}} \overset{d}{\to} -K_{\alpha, \xi},$$
as $n \to \infty$, with $$c_{\alpha, \xi} = \frac{\pi}{2\Gamma(\alpha)\sin(\alpha \pi /2)},$$
where $\Gamma$ is the Gamma function and where $K_{\alpha, \xi}$ is a stable random variable of index $\alpha$, with $\gamma = 0$, $c = 1$ and $\beta = 1$ (see e.g. p.62 in \cite{stableBook}).

For the log-returns $\ln(\rho_1)$, we take a stable distribution with index $1 < \tilde{\alpha} < 2$, and parameters $\tilde{\gamma} = 0$, $\tilde{c} = 1$ and $\tilde{\beta} \in [-1, 1]$. Then, we have $\mu_{\rho} = 0$ and
$$\sum_{k = 1}^n\frac{\ln(\rho_k) - \mu_{\rho}}{c_{\tilde{\alpha}, \rho} n^{1/\tilde{\alpha}}} \overset{d}{\to} K_{\tilde{\alpha}, \rho},$$
as $n \to \infty$. Thus, Theorem \ref{thm_1} implies that $\theta^{(n)} \overset{d}{\to} Y$, as $n \to \infty$, where
$$Y_t = e^{R_t}\left(y + \int_{0+}^t e^{-R_{s-}} dX_s\right), t \geq 0,$$
with $X_t = \mu_{\xi}t + L^{\alpha}_t$ and $R_t = \mu_{\rho}t + L^{\tilde{\alpha}}_t$, where $L^{\alpha}$ and $L^{\tilde{\alpha}}$ are stable L\'{e}vy processes with $L^{\alpha}_1 \overset{d}{=} -K_{\alpha, \xi}$ and $L^{\tilde{\alpha}}_1 \overset{d}{=} K_{\tilde{\alpha}, \rho}$.
\end{ex}

As already mentioned, we will be interested in the application of Theorem \ref{thm_1} to ruin theory and we now state the main consequence for this line of study. Define the following stopping times, for $n \geq 1$,
$$\tau^n(y) = \inf\{t > 0 : \theta^{(n)}_t < 0\}$$ 
with the convention $\inf\{\emptyset\} = +\infty$, and also 
$$\tau(y) = \inf\{t > 0 : Y_t < 0\}.$$ 

\begin{thm}\label{cor_convRuinProb}
Assume that $\mathbf{(H)}$ holds. We have, for all $T \geq 0$,
$$\lim_{n \to \infty}\pP(\tau^n(y) \leq T) = \pP(\tau(y) \leq T)$$
and, equivalently, $\tau^n(y) \overset{d}{\to} \tau(y)$, as $n \to \infty$.
\end{thm}

Theorem \ref{cor_convRuinProb} implies the convergence of $\eE(f(\tau^n(y))$ to $\eE(f(\tau(y))$, for any continuous and bounded function $f : \rR_+ \to \rR$. For example, we can obtain the following convergence result for a simple form of the discounted penalty function.

\begin{coro}\label{cor_convDPF}
Assume that $\mathbf{(H)}$ holds. We have
$$\lim_{n \to \infty}\eE(e^{-\alpha \tau^n(y)}\mathbf{1}_{\{\tau^n(y) < +\infty\}}) = \eE(e^{-\alpha \tau(y)}\mathbf{1}_{\{\tau(y) < +\infty\}}),$$
for all $\alpha > 0$.
\end{coro}

When $\xi_1$ and $\ln(\rho_1)$ both satisfy $(\mathbf{H^2})$, the limiting stable random variable is, in fact, the standard normal random variable and the limiting process is defined by two independent Brownian motions with drift.

\begin{coro}[Pure diffusion limit]\label{thm_2}
Assume that $\xi_1$ and $\ln(\rho_1)$ both satisfy $(\mathbf{H^2})$, then $\theta^{(n)} \overset{d}{\to} Y$, as $n \to \infty$, for $Y = (Y_t)_{t \geq 0}$ defined by $(\ref{eq_GOU})$ with $R_t = \mu_{\rho} t + \sigma_{\rho} W_t$ and $X_t = \mu_{\xi} t + \sigma_{\xi} \tilde{W}_t$, for all $t \geq 0$, where $(W_t)_{t\geq0}$ and $(\tilde{W}_t)_{t\geq0}$ are two independent standard Brownian motions and $\sigma_{\xi}^2 = \mathrm{Var}(\xi_1)$ and $\sigma_{\rho}^2 = \mathrm{Var}(\ln(\rho_1))$.
\end{coro}

\begin{ex}[Pareto losses and NIG log-returns]\label{ex_1par}\rm
To illustrate $\mathbf{(H^2)}$ we take again the negative of a Pareto (type I) distribution for the loss $\xi_1$ but with shape parameter $\alpha > 2$, so that the distribution admits also a second moment. For the log-returns, $\ln(\rho_1)$ we take the normal inverse gaussian $\mathrm{NIG}(\alpha, \beta, \delta, \mu)$ with parameters $0 \leq |\beta| < \alpha$, $\delta > 0$ and $\mu \in \rR$, i.e. the random variable defined by the following moment generating function
\begin{equation}\label{eq_momgen_NIG}
\eE(e^{u\ln(\rho_1)}) = \exp\left(\mu u + \delta\left(\lambda - \sqrt{\alpha^2 - (\beta + u)^2}\right)\right),
\end{equation}
where $\lambda = \sqrt{\alpha^2 - \beta^2}$, for all $u \in \rR$.

Then, it is well known that 
$$\mu_{\xi} = -\frac{\alpha}{\alpha - 1}, \hspace{3mm} \sigma_{\xi}^2 = \frac{\alpha}{(\alpha - 1)^2(\alpha -2)}$$
and that
$$\mu_{\rho} = \mu + \frac{\beta\delta}{\lambda}, \hspace{3mm} \sigma_{\rho}^2 = \delta\frac{\alpha^2}{\lambda^3}.$$
Thus, in this case, Corollary \ref{thm_2} yields $\theta^{(n)} \overset{d}{\to} Y$, with 
\begin{equation*}
Y_t = e^{\mu_{\rho}t+\sigma_{\rho}W_t}\left(y + \int_{0+}^t e^{-\mu_{\rho}s-\sigma_{\rho}W_s} d(\mu_{\xi}s+\sigma_{\xi}\tilde{W}_s)\right), t \geq 0,
\end{equation*}
and where $(W_t)_{t \geq 0}$ and $(\tilde{W}_t)_{t \geq 0}$ are two independent standard Brownian motions.
\end{ex}

\subsection{The UT Condition and the Proofs of Theorems \ref{thm_1} and \ref{cor_convRuinProb}}

We now turn to the proofs of the Theorems. The strategy is to rewrite the discrete-time process as a stochastic integral and to use the well-known weak convergence result for stochastic integrals based on the UT (uniform tightness) condition for semimartingales.

To rewrite the discrete-time process, note that, by induction, the explicit solution of (\ref{eq_AR}), for all $n \in \nN^*$ and $k \in \nN^*$, is given by
\begin{equation*}
\begin{split}
\theta^{(n)}\left(\frac{k}{n}\right) & = y \prod_{i = 1}^k \rho^{(n)}_i + \sum_{i = 1}^k \xi^{(n)}_i \prod_{j = i + 1}^k \rho^{(n)}_j \\
& = \prod_{i = 1}^k \rho^{(n)}_i\left(y + \sum_{i = 1}^k \xi^{(n)}_i \prod_{j = 1}^i (\rho^{(n)}_j)^{-1}\right),
\end{split}
\end{equation*}
where, by convention, we set $\prod_{j = k + 1}^k \rho^{(n)}_j = 1$, for all $n \in \nN^*$. Thus,
\begin{equation}\label{eq_discreteExplicit}
\theta^{(n)}_t = \prod_{i = 1}^{[nt]} \rho^{(n)}_i\left(y + \sum_{i = 1}^{[nt]} \xi^{(n)}_i \prod_{j = 1}^i (\rho^{(n)}_j)^{-1}\right).
\end{equation}
and setting $X^{(n)}_t = \sum_{i = 1}^{[nt]}\xi^{(n)}_i$ and $R^{(n)}_t = \sum_{i = 1}^{[nt]}\gamma^{(n)}_i$, we obtain
\begin{equation}\label{eq_GOUtheta}
\theta^{(n)}_t = e^{R^{(n)}_t}\left(y + \int_{0+}^{t}e^{-R^{(n)}_{s-}}dX^{(n)}_s\right).
\end{equation}
In fact, the above rewriting of the discrete-time process will prove very useful for most proofs in this paper.

\begin{rem}\rm
An other way to prove the weak convergence would be to remark that since $[X^{(n)}, R^{(n)}]_t = 0$, for all $n \in \nN^*$, we find that $\theta^{(n)}$ satisfies the following stochastic differential equation :
$$\theta^{(n)}_t = y + X^{(n)}_{t} + \int_{0+}^t\theta^{(n)}_{s-}d\hat{R}^{(n)}_s,$$
where 
$$\hat{R}^{(n)}_t = R^{(n)} + \sum_{0 < s \leq t}(e^{\Delta R^{(n)}_s} - 1 - \Delta R^{(n)}_s) = \sum_{i = 1}^{[nt]}(e^{\gamma^{(n)}_i} - 1),$$
and to use the well-known stability results for the solutions of stochastic differential equations. We refer to \cite{duffieprotter1992} for an interesting application of this method for different models in mathematical finance. However, this way seems harder, in our case, since the process $(\hat{R}^{(n)}_t)_{t \geq 0}$ is less explicit than $(R^{(n)}_t)_{t \geq 0}$.
\end{rem}

We now recall the UT condition, the weak convergence result and give two lemmas to check the condition in our case.

\begin{defi}
Consider a sequence of real-valued semimartingales $Z^{(n)}$ defined on $(\Omega^{(n)}, \mathcal{F}^{(n)}, (\mathcal{F}^{(n)}_t)_{t \geq 0}, \pP^{(n)})$, for each $n \in \nN^*$. Denote by $\mathcal{H}^{(n)}$ the set given by
\begin{equation*}
\begin{split}
\mathcal{H}^{(n)} = \{H^{(n)} | & H^{(n)}_t = L^{n, 0} + \sum_{i = 1}^p L^{n, i}\mathbf{1}_{[t_i, t_{i+1})}(t), p \in \nN, \\
& 0 = t_0 < t_1 < \dots < t_p = t, \\
& L^{n, i} \text{ is } \mathcal{F}^{(n)}_{t_i}-\text{measurable} \text{ with } |L^{n, i}| \leq 1\}.
\end{split}
\end{equation*}
The sequence $(Z^{(n)})_{n \in \nN^*}$ is UT (also called P-UT in \cite{js2003}, for "uniformly tight" and "predictably uniformly tight") if for all $t > 0$, for all $\epsilon > 0$, there exists $M > 0$ such that, 
$$\sup_{H^{(n)} \in \mathcal{H}^{(n)}, n \in \nN^*}\pP^{(n)} \left(\left|\int_{0+}^tH^{(n)}_{s-}dZ^{(n)}_s\right| > M\right) < \epsilon.$$
\end{defi}

For more information about the UT condition see Section VI.6 in \cite{js2003}. One of the interesting consequences of the UT condition is given by the following proposition which is a particular case of Theorem 6.22 p.383 of \cite{js2003}.

\begin{prop}\label{prop_UTConv}
Let $(H^{(n)}, Z^{(n)})_{n \in \nN^*}$ be a sequence of real-valued semimartingales defined on $(\Omega^{(n)}, \mathcal{F}^{(n)}, (\mathcal{F}^{(n)}_t)_{t \geq 0}, \pP^{(n)})$. If $(H^{(n)}, Z^{(n)}) \overset{d}{\to} (H, Z)$ as $n \to \infty$ and the sequence $(Z^{(n)})_{n \in \nN^*}$ is UT, then $Z$ is a semimartingale and when $n \to \infty$,
$$\left(H^{(n)}, Z^{(n)}, \int_0^.H^{(n)}_{s-}dZ^{(n)}_s\right) \overset{d}{\to} \left(H, Z, \int_0^.H_{s-}dZ_s\right).$$
\end{prop}

The following lemma is based on Remark 6.6 p.377 in \cite{js2003}.

\begin{lem}\label{lem_UT1}
Let $(Z^{(n)})_{n \in \nN^*}$ be a sequence of real-valued semimartingales with locally bounded variation defined on $(\Omega^{(n)}, \mathcal{F}^{(n)}, (\mathcal{F}^{(n)}_t)_{t \geq 0}, \pP^{(n)})$. If for each $t > 0$ and each $\epsilon > 0$, there exists $M > 0$ such that 
$$\sup_{n \in \nN^*}\pP^{(n)}\left(V(Z^{(n)})_t > M\right) < \epsilon,$$ 
where $V(.)$ denotes the total first order variation of a process, then $(Z^{(n)})_{n \geq 1}$ is UT.
\end{lem}

\begin{proof}
For each $n \in \nN^*$, $H^{(n)} \in \mathcal{H}^{(n)}$ and $t > 0$, we find $p \in \nN$ and $0 = t_0 < t_1 < \dots < t_p = t$ such that
\begin{equation*}
\begin{split}
\left|\int_{0+}^tH^{(n)}_{s-}dZ^{(n)}_s\right| & \leq |L^{n, 0}| + \sum_{i = 1}^p|L^{n,i}||Z_{t_{i + 1}} - Z_{t_i}| \leq 1 + \sum_{i = 1}^p|Z_{t_{i + 1}} - Z_{t_i}| \\
& \leq 1 + V(Z^{(n)})_t.
\end{split}
\end{equation*}
Thus, the assumption implies the UT property.
\end{proof}

The following lemma is based on Remark 2-1 in \cite{memin1991}.

\begin{lem}\label{lem_UT2}
Let $(Z^{(n)})_{n \in \nN^*}$ be a sequence of real-valued local martingales defined on $(\Omega^{(n)}, \mathcal{F}^{(n)}, (\mathcal{F}^{(n)}_t)_{t \geq 0}, \pP^{(n)})$ and $Z$ a real-valued semimartingale on  $(\Omega, \mathcal{F}, (\mathcal{F}_t)_{t \geq 0}, \pP)$. Denote by $\nu^{(n)}$ the compensator of the jump measure of $Z^{(n)}$. If $Z^{(n)} \overset{d}{\to} Z$ as $n \to \infty$, then the following conditions are equivalent:
\begin{enumerate}
\item[(i)] $(Z^{(n)})_{n \in \nN^*}$ is UT,
\item[(ii)] for each $t > 0$ and each $\epsilon > 0$, there exists $a, M > 0$ such that
$$\sup_{n \geq 1}\pP^{(n)}\left(\int_0^t \int_{\rR}|x|\mathbf{1}_{\{|x| > a\}}  \nu^{(n)}(ds, dx) > M\right) < \epsilon.$$ 
\end{enumerate}
\end{lem}

\begin{proof}
From Lemma 3.1. in \cite{jakubowski1989} we know that, under the assumption $Z^{(n)} \overset{d}{\to} Z$ as $n \to \infty$, (i) is equivalent to asking that for each $t > 0$ and each $\epsilon > 0$, there exists $M > 0$ such that 
$$\sup_{n \geq 1}\pP^{(n)}(V(B^{a, n})_t > M) < \epsilon,$$ 
where $V(.)$ is the total first order variation of a process and $B^{a, n}$ is the first semimartingale characteristic of $Z^{(n)}$ (for the truncation function $h(x) = x\mathbf{1}_{\{|x| > a\}}$).

Let's compute $V(B^{a, n})$ in this case. For $a > 0$ and $n \in \nN^*$, define $\tilde{Z}^{n, a}_t = Z^{(n)}_t -  \sum_{0 < s \leq t}\Delta Z_s \mathbf{1}_{\{|\Delta Z_s| > a\}}$ and $B^{a, n}_t = \int_0^t\int_{\rR}x\mathbf{1}_{\{|x| > a\}}\nu^{(n)}(ds, dx).$ We have,
\begin{equation*}
\begin{split}
\tilde{Z}^{n, a}_t & = \tilde{Z}^{n, a}_t + B^{a, n}_t - B^{a, n}_t  \\
& = Z^{(n)}_t - \int_0^t\int_{\rR}x\mathbf{1}_{\{|x| > a\}}(\mu^{(n)}(ds, dx) - \nu^{(n)}(ds, dx)) - B^{a, n}_t,
\end{split}
\end{equation*}
where $\mu^{(n)}$ is the jump measure of $Z^{(n)}$. Thus, since the two first terms on the r.h.s. of the last line above are local martingales, their difference is a local martingale with bounded jumps and thus the first semimartingale characteristic of $Z^{(n)}$ is $B^{a, n}_t$.  So, 
$$V(B^{a, n})_t = \int_0^t\int_{\rR}|x|\mathbf{1}_{\{|x| > a\}}\nu^{(n)}(ds, dx)$$ 
and this finishes the proof.
\end{proof}

We are now ready for the proof of Theorem \ref{thm_1}.

\begin{proof}[Proof of Theorem \ref{thm_1}]
To be able to apply Proposition \ref{prop_UTConv}, we need show that $(e^{R^{(n)}}, X^{(n)})_{n \in \nN^*}$ converges in law as $n \to \infty$ and that $(X^{(n)})_{n \in \nN^*}$ is UT.

First, note that by definition of $\gamma_k^{(n)}$, we have 
\begin{equation}\label{eq_decomposeR}
R^{(n)}_t = \sum_{i = 1}^{[nt]}\gamma_k^{(n)} = \mu_{\rho}\frac{[nt]}{n} + \sum_{i = 1}^{[nt]}\frac{\ln(\rho_i) - \mu_{\rho}}{c_{\beta}n^{1/\beta}}.
\end{equation}
But $[nt]/n \to t$ as $n \to \infty$. By the stable functional convergence theorem (see e.g. Theorem 2.4.10 p.95 in \cite{embrechts1997}), the sum in the r.h.s. of the equation above converges weakly to a stable L\'{e}vy process $(L^{\beta}_t)_{t \geq 0}$ with $L^{\beta}_1 \overset{d}{=} K_{\beta}$. Thus, we obtain 
$$(e^{-R^{(n)}_t})_{t \geq 0} = \left(\exp\left(-\sum_{i = 1}^{[nt]}\gamma_k^{(n)}\right)\right)_{t \geq 0} \overset{d}{\to} \left(e^{- \mu_{\rho} t - L^{\beta}_t}\right)_{t \geq 0}.$$ 

Similarly, by the definition of $\xi^{(n)}_i$, we have
\begin{equation}\label{eq_decomposeX}
X^{(n)}_t = \sum_{i = 1}^{[nt]} \frac{\mu_{\xi}}{n} + \sum_{i = 1}^{[nt]} \frac{\xi_i - \mu_{\xi}}{c_{\alpha}n^{1/\alpha}} = \mu_{\xi} A^{(n)}_t + N^{(n)}_t, \text{ for all } t \geq 0.
\end{equation}
Applying the stable functional convergence theorem again, we obtain $(N^{(n)}_t)_{t \geq 0} \overset{d}{\to} (L^{\alpha}_t)_{t \geq 0}$, as $n \to \infty$, where $L^{\alpha}$ is a stable L\'{e}vy motion, with $L^{\alpha}_1 \overset{d}{=} K_{\alpha}$, which is independent of $(L^{\beta}_t)_{t \geq 0}$ since the sequences $(\xi_k)_{k \in \nN^*}$ and $(\rho_k)_{k \in \nN^*}$ are independent. Using the independence, we also have the convergence of the couple $(e^{R^{(n)}}, X^{(n)})$, as $n \to \infty$.

To prove that $(X^{(n)})_{n \in \nN^*}$ is UT, it is enough to prove that $(A^{(n)})_{n \in \nN^*}$ and $(N^{(n)})_{n \in \nN^*}$ are both UT. Note that $A^{(n)}$ is a process of locally bounded variation for each $n \geq 1$ with $V(A^{(n)}) = A^{(n)}$. Since $A^{(n)}_t \leq t$, for all $n \in \nN^*$, we have
$$\sup_{n \geq 1}\pP(A^{(n)}_t > M) \leq \pP(t > M),$$
for all $M > 0$ and thus, by Lemma \ref{lem_UT1}, the sequence $(A^{(n)})_{n \in \nN^*}$ is UT.

Now, note that, when $t > s$ and $[nt] \geq [ns]+1$, using the i.i.d. property of $(\xi_k)_{k \in \nN^*}$ we obtain
$$\eE(N^{(n)}_t - N^{(n)}_s|\mathcal{F}_s) = \sum_{i = [ns]+1}^{[nt]}\eE\left(\frac{\xi_i - \mu_{\xi}}{c_{\alpha}n^{1/\alpha}}\right) = 0.$$
When $t > s$ and $[nt] < [ns]+1$, $N^{(n)}_t - N^{(n)}_s = 0$, and thus $\eE(N^{(n)}_t - N^{(n)}_s|\mathcal{F}_s) = 0$. This shows that $N^{(n)}$ is a local martingale for each $n \in \nN^*$. Then, denoting by $\nu^{(n)}$ the compensator of the jump measure of $N^{(n)}$ (which is deterministic since $N^{(n)}$ is also a semimartingale with independent increments), we set 
$$s_n = \int_0^t \int_{\rR}|x|\mathbf{1}_{\{|x| > 1\}}  \nu^{(n)}(ds, dx),$$
for each $n \in \nN^*$, and we will show that the (deterministic) sequence $(s_n)_{n \in \nN^*}$ converges (and thus is bounded).

First, we have
\begin{equation*}
\begin{split}
& \int_0^t \int_{\rR}|x|\mathbf{1}_{\{|x| > 1\}}  \nu^{(n)}(ds, dx) = \eE\left(\sum_{0 < s \leq t}|\Delta N^{(n)}_s| \mathbf{1}_{\{|\Delta N^{(n)}| \geq 1\}}\right) \\
& = \sum_{i = 1}^{[nt]} \eE\left(\left|\frac{\xi_i - \mu_{\xi}}{c_{\alpha}n^{1/\alpha}}\right| \mathbf{1}_{\left\{\left|\frac{\xi_i - \mu_{\xi}}{c_{\alpha}n^{1/\alpha}}\right| \geq 1\right\}}\right) \\
& = \frac{[nt]}{c_{\alpha}n^{1/\alpha}}\eE\left(\left|\xi_1 - \mu_{\xi}\right| \mathbf{1}_{\{\left|\xi_1 - \mu_{\xi}\right| \geq c_{\alpha}n^{1/\alpha}\}}\right). 
\end{split}
\end{equation*}
To compute the expectation on the r.h.s., note that for any non-negative random variable $Z$ and constant $a \geq 0$ we have
\begin{equation*}
\begin{split}
\eE(Z\mathbf{1}_{\{Z \geq a\}}) & = \eE\left(\int_{0}^Z \mathbf{1}_{\{Z \geq a\}} dx\right) = \eE\left(\int_{0}^{\infty} \mathbf{1}_{\{Z \geq x \vee a\}} dx\right) \\
& = \int_0^{\infty}\pP(Z \geq x \vee a)dx \\
& = a \pP(Z \geq a) + \int_a^{\infty}\pP(Z \geq x)dx.
\end{split}
\end{equation*}
Thus,
\begin{equation*}
\begin{split}
s_n & = \frac{[nt]}{c_{\alpha}n^{1/\alpha}}\eE\left(\left(\xi_1 - \mu_{\xi}\right) \mathbf{1}_{\{\left(\xi_1 - \mu_{\xi}\right) \geq c_{\alpha}n^{1/\alpha}\}}\right) \\
& + \frac{[nt]}{c_{\alpha}n^{1/\alpha}}\eE\left(-\left(\xi_1 - \mu_{\xi}\right) \mathbf{1}_{\{-\left(\xi_1 - \mu_{\xi}\right) \geq c_{\alpha}n^{1/\alpha}\}}\right) \\
& = [nt]\pP\left(\xi_1 \geq \mu_{\xi} + c_{\alpha}n^{1/\alpha}\right) + \frac{[nt]}{c_{\alpha}n^{1/\alpha}} \int_{c_{\alpha}n^{1/\alpha}}^{\infty}\pP\left(\xi_1 \geq \mu_{\xi} + x\right)dx \\
& + [nt]\pP\left(\xi_1 \leq \mu_{\xi} - c_{\alpha}n^{1/\alpha}\right) + \frac{[nt]}{c_{\alpha}n^{1/\alpha}} \int_{c_{\alpha}n^{1/\alpha}}^{\infty}\pP\left(\xi_1 \leq \mu_{\xi} - x\right)dx.
\end{split}
\end{equation*}
Using the fact that $\xi_1$ satisfies $(\mathbf{H^{\alpha}})$, we see that $\pP\left(\xi_1 \leq \mu_{\xi} - c_{\alpha}x^{1/\alpha}\right) \sim k_1^{\xi_1} c_{\alpha}^{-\alpha}x^{-1}$ and $\pP\left(\xi_1 \geq \mu_{\xi} + c_{\alpha}x^{1/\alpha}\right) \sim k_2^{\xi_1}c_{\alpha}^{-\alpha}x^{-1}$, as $x \to \infty$. So,
\begin{equation}\label{eq_thm1_convsn}
\begin{split}
\lim_{n \to \infty}s_n & = \lim_{n \to \infty}\frac{k_1^{\xi_1}}{c_{\alpha}^{\alpha}}\frac{[nt]}{n} - \lim_{n \to \infty}\frac{[nt]}{c_{\alpha}n^{1/\alpha}}\frac{k_1^{\xi_1}c_{\alpha}^{1-\alpha}n^{(1-\alpha)/\alpha}}{1-\alpha} \\
& + \lim_{n \to \infty}\frac{k_2^{\xi_1}}{c_{\alpha}^{\alpha}}\frac{[nt]}{n} - \lim_{n \to \infty}\frac{[nt]}{c_{\alpha}n^{1/\alpha}}\frac{k_2^{\xi_1}c_{\alpha}^{1-\alpha}n^{(1-\alpha)/\alpha}}{1-\alpha} \\
& = \frac{k_1^{\xi_1} + k_2^{\xi_1}}{c_{\alpha}^{\alpha}}\frac{\alpha}{\alpha - 1} t.
\end{split}
\end{equation}

Thus, the sequence is bounded and taking $M > 0$ large enough, we find $\sup_{n \geq 1}\pP(s_n > M) < \epsilon$, for each $\epsilon > 0$, and, by Lemma \ref{lem_UT2}, we have then shown that the sequence $(N^{(n)})_{n \in \nN^*}$ is UT.

To conclude we obtain, using Proposition \ref{prop_UTConv} and the continuous mapping theorem with $h(x_1, x_2, x_3) = (x_3 + y)/x_2$, $(\theta^{(n)}_t)_{t \geq 0} \overset{d}{\to} (Y_t)_{t \geq 0}$ where $Y = (Y_t)_{t \geq 0}$ is given by (\ref{eq_GOU}) with $R_t = \mu_{\rho} t + L^{\beta}_t$, $X_t = \mu_{\xi} t + L^{\alpha}_t$, for all $t \geq 0$.

In this case, we have $[R, X]_t = 0$, for all $t \geq 0$, (see e.g. Theorem 33 and its proof p.301-302 in \cite{protterSDE}) and thus, using It\^{o}'s lemma and Theorem II.8.10 p.136 in \cite{js2003}, we obtain the stochastic differential equation (\ref{eq_SDE}).
\end{proof}

\begin{proof}[Proof of Theorem \ref{cor_convRuinProb}]
We start by proving that $\pP(\inf_{0 \leq t \leq T}Y_t = 0) = 0$. First, note that
$$\left\{\inf_{0 \leq t \leq T}Y_t = 0\right\} = \left\{\sup_{0 \leq t \leq T}\left(-\int_{0+}^te^{-R_{s-}}dX_s\right) = y\right\}.$$
Using the independence of the processes, we then obtain
$$\pP\left(\inf_{0 \leq t \leq T}Y_t = 0\right) = \int_{D}\pP\left(\sup_{0 \leq t \leq T}\left(-\int_{0+}^tg(s-)dX_s\right) = y\right)\pP_{e^{-R}}(dg)$$
where $D$ is the space of \cadlag functions and $\pP_{e^{-R}}$ is the law of the process $(e^{-R_t})_{t \geq 0}$. Denote $S(g)_t = -\int_{0+}^tg(s-)dX_s$, for all $t \geq 0$.

Let $(t_i)_{i \in \nN^*}$ be an enumerating sequence of $[0, T] \cap \qQ$. Since $S(g) = (S(g)_t)_{t \geq 0}$ is a process with independent increments, $S(g)$ has, for each fixed time $t_i > 0$, the same law as a L\'{e}vy process $L = (L_t)_{t \geq 0}$ defined by the characteristic triplet $(a_L, \sigma_L^2, \nu_L)$ with
$$a_L = \frac{\mu_{\xi}}{t_i}\int_0^{t_i}g(s-)ds, \hspace{2mm} \sigma_L^2 = \frac{\sigma_{\xi}^2}{t_i}\int_0^{t_i}g^2(s-)ds$$
and
$$\nu_L(dx) = \frac{\nu_{\xi}(dx)}{t_i}\int_{0}^{t_i}g(s-)ds,$$ 
where $(a_{\xi}, \sigma^2_{\xi}, \nu_{\xi})$ is the characteristic triplet of $X$, see Theorem 4.25 p.110 in \cite{js2003}. Then, it is well known that $L_{t_i}$ admits a density if $\sigma_L^2 > 0$ or $\nu_L(\rR) = \infty$, see e.g. Proposition 3.12 p.90 in \cite{conttankov}. But, when $\xi_1$ satisfies $\mathbf{(H^2)}$, we have $\sigma_{\xi}^2 > 0$ and $\sigma_L^2 > 0$. When $\xi_1$ satisfies $\mathbf{(H^{\alpha})}$, we have $\nu_{\xi}(\rR) = \infty$ and $\nu_L(\rR) = \infty$. Thus, in both cases, $L_{t_i}$ admits a density and we have $\pP(S(g)_{t_i} = y) = \pP(L_{t_i} = y) = 0$.

Since $(S(g)_t)_{t \geq 0}$ is \cadlag we have
$$\sup_{0 \leq t \leq T}S(g)_t = \sup_{t \in [0, T] \cap \qQ}S(g)_t,$$
and, since a \cadlag process reaches its supremum almost surely,
\begin{equation*}
\begin{split}
\pP\left(\sup_{0 \leq t \leq T}S(g)_t = y\right) & = \pP\left(\sup_{t \in [0, T] \cap \qQ}S(g)_t = y\right) \leq \pP\left(\bigcup_{i \in \nN}\{S_{t_i} = y\}\right) \\
& = \lim_{N \to \infty}\pP\left(\bigcup_{i = 1}^N\{S_{t_i} = y\}\right) \leq \lim_{N \to \infty}\sum_{i = 1}^N\pP(S_{t_i} = y) \\
& = 0.
\end{split}
\end{equation*}
Thus, $\pP(\inf_{0 \leq t \leq T}Y_t = 0) = 0$.

Next, note that we have
$$\left\{\inf_{0 \leq t \leq T}Y_t < 0\right\} \subseteq \{\tau(y) \leq T\} \subseteq \left\{\inf_{0 \leq t \leq T}Y_t \leq 0\right\}$$
and 
$$\left\{\inf_{0 \leq t \leq T}\theta^{(n)}_t < 0\right\} \subseteq \{\tau^n(y) \leq T\} \subseteq \left\{\inf_{0 \leq t \leq T}\theta^{(n)}_t \leq 0\right\}.$$
Since $\theta^{(n)} \overset{d}{\to} Y$ by Theorem \ref{thm_1}, we obtain from the continuous mapping theorem that $\inf_{0 \leq t \leq T}\theta^{(n)}_t \overset{d}{\to} \inf_{0 \leq t \leq T}Y_t$, for all $T \geq 0$, since the supremum (and also the infimum) up to a fixed time are continuous for the Skorokhod topology (see e.g. Proposition 2.4, p.339, in \cite{js2003}). So, by the portmanteau theorem,
\begin{equation*}
\begin{split}
\limsup_{n \to \infty}\pP(\tau^n(y) \leq T) & \leq \limsup_{n \to \infty}\pP\left(\inf_{0 \leq t \leq T}\theta^{(n)}_t \leq 0\right) \\
&  \leq \pP\left(\inf_{0 \leq t \leq T}Y_t \leq 0\right) = \pP\left(\inf_{0 \leq t \leq T}Y_t < 0\right) \\
& = \pP(\tau(y) \leq T),
\end{split}
\end{equation*}
and 
\begin{equation*}
\begin{split}
\liminf_{n \to \infty}\pP(\tau^n(y) \leq T) & \geq \liminf_{n \to \infty}\pP\left(\inf_{0 \leq t \leq T}\theta^{(n)}_t < 0\right) \\
& \geq \pP\left(\inf_{0 \leq t \leq T}Y_t < 0\right) = \pP\left(\inf_{0 \leq t \leq T}Y_t \leq 0\right) \\
& = \pP(\tau(y) \leq T).
\end{split}
\end{equation*}
\end{proof}

\section{Convergence and Approximation of the Ruin Functionals in the Pure Diffusion Case}

In this section, we obtain sufficient conditions for the convergence of a simple form of the discounted penalty function, the ultimate ruin probability and the moments and give a manner to approximate these quantities. To be able to go further (and to obtain practical expressions for the ruin functionals of the limiting process), we now restrict ourselves to the $\mathbf{(H^2)}$ case.

\begin{asm}[$\mathbf{H'}$]\rm
We assume that $\xi_1$ and $\ln(\rho_1)$ both satisfy $(\mathbf{H^2})$. So $Y$ is given by (\ref{eq_GOU}) with $X_t = \mu_{\xi}t + \sigma_{\xi}\tilde{W}_t$ and $R_t = \mu_{\rho}t + \sigma_{\rho}W_t$ or, equivalently, is given by the solution of $(\ref{eq_SDE})$ with the same $X$ and $\hat{R}_t = \kappa_{\rho} t + \sigma_{\rho} W_t$ and $\kappa_{\rho} = \mu_{\rho} + \sigma_{\rho}^2/2$. 
\end{asm}

\subsection{Approximation of the Discounted Penalty Function}

We have seen in Corollary \ref{cor_convDPF} that a simple form of the discounted penalty function converges. In this section, we give an expression of this quantity for the limiting process which will depend on the solution of a second order ODE.

\begin{lem}\label{lem_ODE}
Let $\alpha > 0$. The equation
\begin{equation}\label{eq_ODE}
(\sigma_{\xi}^2 + \sigma_{\rho}^2 x^2)f_{\alpha}''(x) + 2(\mu_{\xi} + \kappa_{\rho} x) f_{\alpha}'(x) - 2 \alpha f_{\alpha}(x) = 0,
\end{equation}
admits a solution $f_{\alpha} : \rR_+ \to \rR$ satisfying 
\begin{itemize}
\item[(P)] $f_{\alpha}(x) > 0$, for all $x \in \rR_+$, and $f'_{\alpha}(x) \leq 0$, for all $x \in (0, +\infty)$.
\end{itemize} 
Moreover, if $\mu_{\rho} \leq 0$, every other solution $\tilde{f}_{\alpha}$ of (\ref{eq_ODE}) satisfying (P) is given by $\tilde{f}_{\alpha}(x) = Kf_{\alpha}(x)$, for all $x \in \rR_+$, for some constant $K \in \rR_+^*$.
\end{lem}

\begin{proof}
Define 
\begin{equation*}
\begin{split}
p(x) & = \exp\left(2\int_0^x \frac{\mu_{\xi} + \kappa_{\rho} z}{\sigma_{\xi}^2 + \sigma_{\rho}^2 z^2}dz\right) \\
& = \exp\left(\frac{2\mu_{\xi}}{\sigma_{\xi} \sigma_{\rho}}\arctan\left(\frac{\sigma_{\rho}}{\sigma_{\xi}}x\right)\right)\left(1+\frac{\sigma_{\rho}^2}{\sigma_{\xi}^2}x^2\right)^{\kappa_{\rho}/\sigma_{\rho}^2},
\end{split}
\end{equation*}
and 
$$g(x) = -\frac{2\alpha}{\sigma_{\xi}^2 + \sigma_{\rho}^2 x^2},$$
for all $x \in \rR_+$. Then, we can rewrite (\ref{eq_ODE}) in the Sturm-Liouville form
$$\left(p(x)f_{\alpha}'(x)\right)' + p(x)g(x)f_{\alpha}(x) = 0.$$
The existence of a (principal) solution satisfying (P) then follows form Corollary 6.4. p.357 in \cite{hartman2002}. The fact that the solutions are uniquely determined up to a constant factor follows from $\int_1^{\infty}p(x)^{-1}dx = \infty$ and Exercise 6.7. p.358 in \cite{hartman2002}. 
\end{proof}

\begin{rem}\rm
Under the condition $\alpha > \kappa_{\rho}$, it is possible to obtain an explicit solution of $(\ref{eq_ODE})$ using the method of contour integration as was done in Theorem A.1 in \cite{paulsen1997b}. Otherwise, the ODE can be solved using numerical integration.
\end{rem}

We now prove the approximation result.

\begin{thm}
Assume that $(\mathbf{H'})$ holds and that $\mu_{\rho} \leq 0$. Let $\alpha > 0$ and let $f_{\alpha} : \rR_+ \to \rR$ be any solution of (\ref{eq_ODE}) satisfying (P). We have
$$\lim_{n \to \infty}\eE(e^{-\alpha \tau^n(y)}\mathbf{1}_{\{\tau^n(y) < +\infty\}}) = \eE(e^{-\alpha \tau(y)}\mathbf{1}_{\{\tau(y) < +\infty\}}) =  \frac{f_{\alpha}(y)}{f_{\alpha}(0)}.$$
\end{thm}

\begin{proof}
The convergence of the discounted penalty function is the content of Corollary \ref{cor_convDPF}.

We now compute the value for the limiting process using the idea in the proof of Theorem 2.1. in \cite{paulsen1993}. First, we show that $L = (f_{\alpha}(Y_{t \wedge \tau(y)})e^{-\alpha(t \wedge \tau(y))})_{t \geq 0}$ is a martingale with respect to the natural filtration of $Y$. Using It\^{o}'s lemma and the fact that $\langle Y, Y \rangle_t = \sigma_{\xi}^2 t + \sigma_{\rho}^2 \int_0^t Y^2_s ds$, we obtain
\begin{equation*}
\begin{split}
f_{\alpha}(Y_{t \wedge \tau(y)})e^{-\alpha(t \wedge \tau(y))} = f_{\alpha}(y) & + \sigma_{\xi} N^{(1)}_t + \sigma_{\rho} N^{(2)}_t \\
& + \int_0^{t \wedge \tau(y)}\frac{e^{-\alpha s}}{2}I(Y_s)ds,
\end{split}
\end{equation*}
where
$$N^{(1)}_t = \int_0^{t \wedge \tau(y)}e^{-\alpha s} f_{\alpha}'(Y_s) d\tilde{W}_s = \int_0^t \mathbf{1}_{\{s \leq \tau(y)\}} e^{-\alpha s} f_{\alpha}'(Y_s) d\tilde{W}_s,$$
$$N^{(2)}_t = \int_0^{t \wedge \tau(y)}e^{-\alpha s} f_{\alpha}'(Y_s)Y_s dW_s = \int_0^t \mathbf{1}_{\{s \leq \tau(y)\}} e^{-\alpha s} f_{\alpha}'(Y_s)Y_s dW_s$$
and
$$I(Y_s) = (\sigma_{\xi}^2 + \sigma_{\rho}^2 Y_s^2)f_{\alpha}''(Y_s) + 2(\mu_{\xi} + \kappa_{\rho} Y_s)f_{\alpha}'(Y_s) - 2\alpha f_{\alpha}(Y_s).$$
Since $\mathbf{1}_{\{s \leq \tau(y)\}}$ is adapted to the natural filtration of $Y$, $N^{(1)}$ and $N^{(2)}$ are local martingales and since $f_{\alpha}$ solves (\ref{eq_ODE}) $L$ is also a local martingale. 

Note that $Y_{t \wedge \tau(y)} \geq 0$ $(\pP-a.s.)$, for all $t \geq 0$, and that $f_{\alpha}$ is non-increasing by (P). Thus, we have $f_{\alpha}(Y_{t \wedge \tau(y)}) \leq f_{\alpha}(0)$ and we find that $L$ is a bounded local martingale, and thus a martingale. Using the property of constant expectation, we then obtain, for $t \geq 0$,
$$f_{\alpha}(y) = \eE\left(f_{\alpha}(Y_{t \wedge \tau(y)})e^{-\alpha(t \wedge \tau(y))}\right)$$
or equivalently
$$f_{\alpha}(y) = \eE\left(f_{\alpha}(Y_{t \wedge \tau(y)})e^{-\alpha(t \wedge \tau(y))}\mathbf{1}_{\{\tau(y) < +\infty\}}\right) + e^{-\alpha t}\eE\left(f_{\alpha}(Y_t)\mathbf{1}_{\{\tau(y) = \infty\}}\right).$$
Again since $f_{\alpha}(Y_{t \wedge \tau(y)}) \leq f_{\alpha}(0)$, we can pass to the limit $t \to \infty$ in the first expectation. Similarly, on the event $\{\tau(y) = \infty\}$, we have $Y_t \geq 0$ $(\pP-a.s.)$ and $f_{\alpha}(Y_t) \leq f_{\alpha}(0)$, for all $t \geq 0$, and the second term goes to $0$ as $t \to \infty$. Finally, using the fact that $Y$ is an almost surely continuous process, we obtain $Y_{\tau(y)} = Y_{\tau(y)-} = 0$ $(\pP-a.s.)$ and
$$f_{\alpha}(y) = \eE\left(f_{\alpha}(Y_{\tau(y)})e^{-\alpha\tau(y)}\mathbf{1}_{\{\tau(y) < +\infty\}}\right) = f_{\alpha}(0)\eE\left(e^{-\alpha\tau(y)}\mathbf{1}_{\{\tau(y) < +\infty\}}\right).$$

Lemma \ref{lem_ODE} also guarantees that this result does not depend on the particular choice of the solution of (\ref{eq_ODE}).
\end{proof}

\subsection{Approximation of the Ultimate Ruin Probability}

We have seen that, when $\xi_1$ and $\ln(\rho_1)$ both satisfy $(\mathbf{H^2})$, we have
$$\lim_{n \to \infty}\pP(\tau^n(y) \leq T) = \pP(\tau(y) \leq T),$$ 
for all $T \geq 0$. We would like to replace the finite-time ruin probability with the ultimate ruin probability  $\pP(\tau(y) < \infty)$  since for the latter, an explicit expression exists for the limiting process. However, the following classic example (see e.g. \cite{grandell1977}) shows that the ultimate ruin probability may fail to converge even if the finite-time ruin probability does. In fact, take $(Z^{(n)})_{t \geq 0}$ to be the deterministic process defined by
$$Z^{(n)}_t =
\begin{cases}
0 \text{ if } t < n, \\
-1 \text{ if } t \geq n. 
\end{cases}
$$
Then, we have $Z^{(n)} \to Z$, as $n \to \infty$, where $Z_t = 0$, for all $t \geq 0$, and we have also convergence of the finite-time ruin probability, since, as $n \to \infty$, $\inf_{0 \leq t \leq T}Z^{(n)}_t \to 0$, for all $T > 0$. But $\inf_{0 \leq t < \infty}Z^{(n)}_{t} = -1$, for all $n \in \nN^*$, and so the ultimate ruin probability fails to converge.

In general, proving the convergence of the ultimate ruin probability is a hard problem and depends on the particular model (see \cite{grandell1977} for another discussion). Still, we can give a sufficient condition for this convergence.

\begin{thm}\label{thm_convURP}
Assume that $(\mathbf{H'})$ holds. When $\mu_{\rho} \leq 0$, we have 
$$\lim_{n \to \infty}\pP(\tau^n(y) < \infty) = 1.$$ 
When $\mu_{\rho} > 0$, we assume additionally that there exists $C < 1$ and $n_0 \in \nN^{*}$ such that 
\begin{equation}\label{asm_condURP}
\sup_{n \geq n_0}\eE\left(e^{-2\gamma^{(n)}_1}\right)^n = \sup_{n \geq n_0}\eE\left((\rho^{(n)}_1)^{-2}\right)^n \leq C.
\end{equation}
Then,
$$\lim_{n \to \infty}\pP(\tau^n(y) < \infty) = \pP(\tau(y) < \infty) = \frac{H(-y)}{H(0)}$$
where, for $x \leq 0$,
$$H(x) = \int_{-\infty}^{x}(\sigma_{\xi}^2 + \sigma_{\rho}^2z^2)^{-(1/2+\mu_{\rho}/\sigma_{\rho}^2)}\exp\left(\frac{2\mu_{\xi}}{\sigma_{\xi}\sigma_{\rho}}\arctan\left(\frac{\sigma_{\rho}}{\sigma_{\xi}}z\right)\right)dz.$$
\end{thm}

Before turning to the proof of the theorem, we give two examples to illustrate Condition (\ref{asm_condURP}).

\begin{ex}[Approximation of the ruin probability with normal log-returns]\label{ex_normal}
Take $\xi_1$ to be any random variable satisfying $\mathbf{(H^2)}$ and $\ln(\rho_1) \overset{d}{=} \mathcal{N}(\mu_{\rho}, \sigma_{\rho}^2)$, with $\mu_{\rho} > 0$, then 
$$\eE\left(e^{-2\gamma^{(n)}_1}\right)^n = e^{-2(\mu_{\rho} - \sigma_{\rho}^2)},$$
for all $n \in \nN^*$, so $n_0 = 1$ and the condition $C < 1$ is equivalent to $\mu_{\rho} > \sigma_{\rho}^2$.
\end{ex}

\begin{ex}[Approximation of the ruin probability with NIG log-returns]\rm
More generally, take $\xi_1$ to be any random variable satisfying $\mathbf{(H^2)}$ and $\ln(\rho_1)$ to be a normal inverse gaussian $\mathrm{NIG}(\alpha, \beta, \delta, \mu)$ random variable with $0 \leq |\beta| < \alpha$, $\delta > 0$ and $\mu \in \rR$ (recall Example \ref{ex_1par} for the definition). We can then use Taylor's formula on the function $x \mapsto \sqrt{\alpha^2 - (\beta - x)^2}$ around $0$, to obtain 
\begin{equation}\label{eq_taylorNIG}
\sqrt{\alpha^2 - \left(\beta - \frac{2}{\sqrt{n}}\right)^2} = \lambda + \frac{2}{\sqrt{n}}\frac{\beta}{\lambda} - \frac{2}{n}\frac{\alpha^2}{[\alpha^2 - (\beta -x_n)^2]^{3/2}},
\end{equation}
for some $x_n \in [0, 2/\sqrt{n}]$, where $\lambda = \sqrt{\alpha^2 - \beta^2}$. Since the mean is given by $\mu_{\rho} = \mu + \delta \beta/\lambda$, we obtain using (\ref{eq_momgen_NIG}) and (\ref{eq_taylorNIG})
\begin{equation*}
\begin{split}
\lim_{n \to \infty}\eE\left(e^{-2\gamma^{(n)}_1}\right)^n & = \exp\left(-2\mu_{\rho} + \frac{2\delta \alpha^2}{\lambda^3}\right) \\
& = \exp\left(-2\left(\mu + \frac{\delta\beta\lambda^2 - \delta\alpha^2}{\lambda^3}\right)\right).
\end{split}
\end{equation*}
Thus, when 
$$\mu + \frac{\delta\beta\lambda^2 - \delta\alpha^2}{\lambda^3} > 0,$$
this limit is strictly smaller than $1$ and we can find $n_0 \in \nN^*$ and $C < 1$ such that (\ref{asm_condURP}) is satisfied. Taking $\beta = 0$ and $\sigma^2 = \delta/\alpha$ we retrieve the condition for normal returns given in Example \ref{ex_normal}.
\end{ex}

\begin{proof}[Proof of Theorem \ref{thm_convURP}]
We have, for all $n \in \nN^*$ and $T > 0$, 
$$\pP(\tau^n(y) < \infty) \geq \pP(\tau^n(y) \leq T)$$ 
and, by Theorem \ref{cor_convRuinProb},
$$\liminf_{n \to \infty}\pP(\tau^n(y) < \infty) \geq \pP(\tau(y) \leq T).$$
So, letting $T \to \infty$,
$$\liminf_{n \to \infty}\pP(\tau^n(y) < \infty) \geq \pP(\tau(y) < \infty).$$
Now if $\pP(\tau(y) < \infty) = 1$, which is equivalent to $\mu_{\xi} \leq 0$ by \cite{paulsen1998}, there is nothing else to prove. So we assume that $\pP(\tau(y) < \infty) < 1$, or $\mu_{\xi} > 0$, and we will prove that
$$\limsup_{n \to \infty}\pP(\tau^n(y) < \infty) \leq \pP(\tau(y) < \infty),$$
under the additional condition (\ref{asm_condURP}).

Fix $y > \epsilon > 0$, $T > 0$ and, when $\tau^n(y) > T$, denote by $K^{(n)}_{\epsilon, T}$ the event
$$K^{(n)}_{\epsilon, T} = \left\{\left|\int_{T+}^{\tau^n(y)}e^{-R^{(n)}_{s-}}dX^{(n)}_s\right| < \epsilon\right\}.$$
We have,
\begin{equation*}
\begin{split}
\{\tau^n(y) < \infty\} = \{\tau^n(y) \leq T\} & \cup \{\tau^n(y) \in (T, \infty), K^{(n)}_{\epsilon, T}\} \\
& \cup \{\tau^n(y) \in (T, \infty), (K_{\epsilon, T}^{(n)})^{\complement}\}.
\end{split}
\end{equation*}
But, on the event $\{\tau^n(y) \in (T, \infty), K_{\epsilon, T}^{(n)}\}$, 
$$\int_{0+}^T e^{-R^{(n)}_{s-}}dX^{(n)}_s + \int_{T+}^{\tau^n(y)} e^{-R^{(n)}_{s-}}dX^{(n)}_s < -y$$
which implies 
$$\int_{0+}^T e^{-R^{(n)}_{s-}}dX^{(n)}_s < - y + \epsilon,$$
or equivalently that $\tau^n(y-\epsilon) \leq T$, by (\ref{eq_GOUtheta}). Thus,
$$\{\tau^n(y) \leq T\} \cup \{\tau^n(y) \in (T, \infty), K_{\epsilon, T}^{(n)}\} \subseteq \{\tau^n(y-\epsilon) \leq T\}.$$

Then, we have $\{\tau^n(y) \in (T, \infty), (K_{\epsilon, T}^{(n)})^{\complement}\} \subseteq (K_{\epsilon, T}^{(n)})^{\complement}$ and thus
$$\limsup_{n \to \infty}\pP(\tau^n(y) < \infty) \leq \pP(\tau(y-\epsilon) \leq T) + \limsup_{n \to \infty}\pP\left((K_{\epsilon, T}^{(n)})^{\complement}\right).$$
So, we need to show that
$$\lim_{T \to \infty}\limsup_{n \to \infty}\pP\left((K_{\epsilon, T}^{(n)})^{\complement}\right) = 0.$$

Using the decomposition (\ref{eq_decomposeX}), we obtain 
\begin{equation*}
(K^{(n)}_{\epsilon, T})^{\complement} \subseteq \left\{|\mu_{\xi}|\left|\int_{T+}^{\tau^n(y)}e^{-R^{(n)}_{s-}}dA^{(n)}_s \right| \geq \frac{\epsilon}{2}\right\} \cup \left\{\left|\int_{T+}^{\tau^n(y)}e^{-R^{(n)}_{s-}}dN^{(n)}_s \right| \geq \frac{\epsilon}{2}\right\}
\end{equation*}
Denote by $E_{1,T}^{(n)}$ and $E_{2,T}^{(n)}$ the sets on the r.h.s. of the above equation. 

When $n \geq n_0$, we obtain, recalling the explicit form of the integral and using Markov's inequality,
\begin{equation*}
\begin{split}
\pP(E_{1,T}^{(n)}) & \leq \frac{2 |\mu_{\xi}|}{n\epsilon}\eE\left(\sum_{i = [nT]+1}^{[n\tau^n(y)]+1}e^{-\sum_{j = 1}^{i}\gamma^{(n)}_j}\right) \\
& \leq \frac{2 |\mu_{\xi}|}{n\epsilon}\eE\left(\sum_{i = [nT]+1}^{\infty}\prod_{j = 1}^{i}e^{-\gamma^{(n)}_j}\right) = \frac{2 |\mu_{\xi}|}{n\epsilon}\sum_{i = [nT]+1}^{\infty}\eE\left(e^{-\gamma^{(n)}_1}\right)^{i} \\
& = \frac{2 |\mu_{\xi}|}{n\epsilon}\eE\left(e^{-\gamma^{(n)}_1}\right)^{[nT]}\sum_{j = 1}^{\infty}\eE\left(e^{-\gamma^{(n)}_1}\right)^{j}.
\end{split}
\end{equation*}
But, since $\eE(e^{-\gamma^{(n)}_1}) \leq \eE(e^{-2\gamma^{(n)}_1})^{1/2} \leq C^{1/(2n)} < 1$, we have
$$\pP(E_{1,T}^{(n)}) \leq \frac{2 |\mu_{\xi}|}{\epsilon}\frac{C^{1/(2n)}}{n(1- C^{1/(2n)})} C^T.$$
Moreover it is easy to see that $C^{-1/(2n)}(n(1 - C^{-1/(2n)}))^{-1} \to -2/\ln(C)$ as $n \to \infty$, and so $\lim_{T \to \infty}\limsup_{n \to \infty} \pP(E^{(n)}_{1,T}) = 0$.

On the other hand, using the Chebyshev and Burkholder-Davis-Gundy inequalities, we obtain
\begin{equation*}
\begin{split}
\pP(E^{(n)}_{2, T}) & \leq \frac{4}{\epsilon^2}\eE\left(\left|\int_{T+}^{\tau^n(y)}e^{-R^{(n)}_{s-}}dN^{(n)}_s\right|^2\right) \\
& \leq \frac{4}{\epsilon^2}\eE\left(\sup_{T < t < \infty}\left|\int_{T+}^{t}e^{-R^{(n)}_{s-}}dN^{(n)}_s\right|^2\right) \\
& \leq \frac{4K}{\epsilon^2}\eE\left(\int_{T+}^{\infty}e^{-2R^{(n)}_{s-}}d[N^{(n)},N^{(n)}]_s\right),
\end{split}
\end{equation*}
where $K$ is a constant. But,
$$[N^{(n)},N^{(n)}]_t = \sum_{0 < s \leq t}(\Delta N^{(n)}_s)^2 = \sum_{i = 1}^{[nt]}\left(\frac{\xi_i - \mu_{\xi}}{\sqrt{n}}\right)^2.$$
Thus, writing the stochastic integral explicitly and using the same computation as before, we obtain
\begin{equation*}
\begin{split}
\pP(E^{(n)}_{2, T}) & \leq \frac{4K}{\epsilon^2} \eE\left(\sum_{i = [nT] + 1}^{\infty}\left(\frac{\xi_i-\mu_{\xi}}{\sqrt{n}}\right)^2e^{-2\sum_{j = 1}^i\gamma^{(n)}_j}\right) \\
& = \frac{4K\sigma_{\xi}^2}{\epsilon^2 n} \sum_{i = [nT]+1}^{\infty}\eE(e^{-2\gamma^{(n)}_1})^i \leq \frac{4K}{\epsilon^2} C^T \sigma_{\xi}^2 \frac{C^{1/n}}{n(1-C^{1/n})}.
\end{split}
\end{equation*}
Again, using the fact that the expression on the r.h.s. above converges, when $n \to \infty$, we find that
$$\lim_{T \to \infty}\limsup_{n \to \infty}\pP\left(E^{(n)}_{2,T}\right) = 0$$
and
$$\limsup_{n \to \infty}\pP(\tau^n(y) < \infty) \leq \pP(\tau(y-\epsilon) < \infty).$$
So, letting $\epsilon \to 0$ and using the continuity of $y \mapsto \pP(\tau(y) < \infty)$, we obtain
$$\limsup_{n \to \infty}\pP(\tau^n(y) < \infty) \leq \pP(\tau(y) < \infty).$$
The explicit expression for the ultimate ruin probability of the limiting process is given in \cite{paulsen1997b}.
\end{proof}

\subsection{Approximation of the Moments} 

In this section, we obtain a recursive formula for the computation of the moments of the limiting process $Y$ at a fixed time which, for simplicity, we choose to be $T = 1$ and prove the convergence of the moments of $\theta^{(n)}_1$ to the moments of $Y_1$. This gives a way to approximate the moments of $\theta^{(n)}_1$.

\begin{prop}\label{prop_finiteMom}
Assume that the limiting process $Y = (Y_t)_{t \geq 0}$ is given by (\ref{eq_GOU}) with $X_t = \mu_{\xi}t + \sigma_{\xi}\tilde{W}_t$ and $R_t = \mu_{\rho}t + \sigma_{\rho}W_t$, for all $t \geq 0$. We have, for all $p \in \nN$,
\begin{equation}\label{eq_finitemoments}
\eE\left(\sup_{0 \leq t \leq 1}|Y_t|^p\right) < \infty.
\end{equation}
Moreover, letting $m_p(t) = \eE[(Y_t)^p]$, for each  $0 \leq t \leq 1$ and $p \in \nN$, we have the following recursive formula:  $m_0(t) = 1$,
\begin{equation} \label{eq_recursiveMom1}
m_1(t) =
\begin{cases}
y e^{\kappa_{\rho} t} + \frac{\mu_{\xi}}{\kappa_{\rho}}(e^{\kappa_{\rho} t} - 1) \text{ when }\kappa_{\rho} \neq 0, \\
 y + \mu_{\xi}t \text{ when }\kappa_{\rho} = 0,
\end{cases}
\end{equation}
with $\kappa_{\rho} = \mu_{\rho} + \sigma_{\rho}^2/2$ and, for each $p \geq 2$,
\begin{equation} \label{eq_recursiveMom2}
m_p(t) = y^pe^{a_pt} + \int_0^t e^{a_p (t-s)}\left(b_p m_{p-1}(s) + c_p m_{p-2}(s)\right)ds,
\end{equation}
with $a_p = p\mu_{\rho} + p^2\sigma_{\rho}^2/2$, $b_p = p\mu_{\xi}$ and $c_p = p(p-1)\sigma_{\xi}^2/2$.
\end{prop}

\begin{proof}
The existence of the moments (\ref{eq_finitemoments}) follows, for $p \geq 2$, from the general existence result for the strong solutions of SDEs, see e.g. Corollary 2.2.1 p.119 in \cite{nualart} and, for $p = 1$, from Cauchy-Schwarz's inequality. 

Set $m_p(t) = \eE[(Y_t)^p]$, for all $0 \leq t \leq 1$ and $p \in \nN^*$. Suppose that $p \geq 2$. For $r \in \nN^*$, define the stopping times
$$\theta_r = \inf{\{t > 0 : |Y_t| > r\}}$$
with $\inf\{\emptyset\} = +\infty$. Then, applying It\^{o}'s lemma and using $\langle Y, Y \rangle_t = \sigma_{\xi}^2 t + \sigma_{\rho}^2 \int_0^t Y^2_s ds$, yields
\begin{equation*}
\begin{split}
(Y_{t\wedge\theta_r})^p = y^p & + p\mu_{\xi}\int_0^{t\wedge\theta_r}(Y_s)^{p-1}ds + p\sigma_{\xi}\int_0^{t\wedge\theta_r}(Y_s)^{p-1}d\tilde{W}_s \\
& + p\kappa_{\rho}\int_0^{t\wedge\theta_r}(Y_s)^pds + p\sigma_{\rho}\int_0^{t\wedge\theta_r}(Y_s)^pdWs \\
& + \frac{p(p-1)}{2}\sigma_{\xi}^2\int_0^{t\wedge\theta_r}(Y_s)^{p-2} ds \\
& + \frac{p(p-1)}{2}\sigma_{\rho}^2\int_0^{t\wedge\theta_r}(Y_s)^p ds.
\end{split}
\end{equation*}

Thus, using Fubini's theorem and the fact that the stochastic integrals are martingales, we obtain
\begin{equation*}
\begin{split}
\eE[(Y_{t\wedge\theta_r})^p] = y^p & + p\mu_{\xi}\int_0^{t\wedge\theta_r}\eE[(Y_s)^{p-1}]ds + p\kappa_{\rho}\int_0^{t\wedge\theta_r}\eE[(Y_s)^p]ds \\
& + \frac{p(p-1)}{2}\sigma_{\xi}^2\int_0^{t\wedge\theta_r}\eE[(Y_s)^{p-2}] ds \\
& + \frac{p(p-1)}{2}\sigma_{\rho}^2\int_0^{t\wedge\theta_r}\eE[(Y_s)^p] ds.
\end{split}
\end{equation*}

Now we can take the limit as $r \to \infty$, and use (\ref{eq_finitemoments}) to pass it inside the expectation of the l.h.s. of the above equation. Differentiating w.r.t. $t$, we then obtain the following ODE
\begin{equation*}
\begin{split}
\frac{d}{dt}\eE[(Y_t)^p] & = \left(p\kappa_{\rho} + \frac{p(p-1)}{2}\sigma_{\rho}^2\right)\eE[(Y_t)^p] + p \mu_{\xi} \eE[(Y_t)^{p-1}] \\
& + \frac{p(p-1)}{2}\sigma_{\xi}^2\eE[(Y_t)^{p-2}],
\end{split}
\end{equation*}
and $\eE[(Y_0)^p] = y^p$. This is an inhomogeneous linear equation of the first order which can be solved explicitly to obtain (\ref{eq_recursiveMom2}).

For $p = 1$, using the same technique as above, we obtain
$$\eE(Y_t) = y + \mu_{\xi} t + \kappa_{\rho} \int_0^t \eE(Y_s) ds.$$
If $\kappa_{\rho} = 0$, there is nothing to prove. If $\kappa_{\rho} \neq 0$, we obtain by differentiating w.r.t. $t$,
$$\frac{d}{dt}\eE(Y_t) = \mu_{\xi} + k_{\rho} \eE(Y_t),$$
with $\eE(Y_0) = y$ and this can be solved to obtain (\ref{eq_recursiveMom1}).
\end{proof}

We now state the approximation result.

\begin{thm}\label{thm_convMom}
Assume that $(\mathbf{H'})$ holds. Assume that $\eE(|\xi_1|^q) < \infty$, and that 
\begin{equation}\label{asm_thm2}
\sup_{n \in \nN^*}\eE\left(e^{q\gamma^{(n)}_1}\right)^n = \sup_{n \in \nN^*}\eE\left((\rho^{(n)}_1)^q\right)^n < \infty,
\end{equation}
for some integer $q \geq 2$. Then, for each $p \in \nN^*$ such that $1 \leq p < q$, we have 
$$\lim_{n \to \infty}\eE[(\theta^{(n)}_1)^p] = \eE[(Y_1)^p] = m_p(1),$$
for the function $m_p$ defined in Proposition \ref{prop_finiteMom}.
\end{thm}

Before turning to the proof of the theorem, we give an example to illustrate Condition (\ref{asm_thm2}).

\begin{ex}[Approximation of the moments with NIG log-returns]\label{ex_convMomNIG}\rm
Take $\ln(\rho_1)$ to be a normal inverse gaussian $\mathrm{NIG}(\alpha, \beta, \delta, \mu)$ random variable with $0 \leq |\beta| < \alpha$, $\delta > 0$ and $\mu \in \rR$ (recall Example \ref{ex_1par} for the definition). Fix $q \geq 2$. We can use Taylor's formula on the function $x \mapsto \sqrt{\alpha^2 - (\beta + x)^2}$ around $0$, to obtain 
\begin{equation*}
\sqrt{\alpha^2 - \left(\beta + \frac{q}{\sqrt{n}}\right)^2} = \lambda - \frac{q}{\sqrt{n}}\frac{\beta}{\lambda} - \frac{q^2}{2n}\frac{\alpha^2}{[\alpha^2 - (\beta + x_n)^2]^{3/2}},
\end{equation*}
for some $x_n \in [0, q/\sqrt{n}]$, where $\lambda = \sqrt{\alpha^2 - \beta^2}$ and, recalling (\ref{eq_momgen_NIG}),
$$\lim_{n \to \infty}\eE\left(e^{q\gamma^{(n)}_1}\right)^n = \exp\left(q\mu_{\rho} - \frac{q\delta \alpha^2}{2\lambda^3}\right).$$
Since this limit exists and is finite for each $q \geq 2$, the sequence is bounded and the convergence of the moments depends only on the highest moment of $\xi_1$. Note that the NIG distribution contains the standard Gaussian as a particular case. 
\end{ex}

\begin{proof}[Proof of Theorem \ref{thm_convMom}]
Since by Corollary \ref{thm_2}, we know that $\theta^{(n)}_1 \overset{d}{\to} Y_1$, we have also $(\theta^{(n)}_1)^p \overset{d}{\to} (Y_1)^p$, as $n \to \infty$, for $1 \leq p < q$. It is thus enough to show that the sequence $((\theta^{(n)}_1)^p)_{n \in \nN^*}$ is uniformly integrable, which by de la Vall\'{e}e-Poussin's criterion is implied by the condition $\sup_{n \in \nN^*}\eE(|\theta^{(n)}_1|^{q}) < \infty$.

Define $\tilde{R}^{(n)}_t = R^{(n)}_1 - R^{(n)}_{1 - t}$ and $\tilde{X}^{(n)}_t = X^{(n)}_1 - X^{(n)}_{1 - t}$ the time-reversed processes of $R^{(n)}$ and $X^{(n)}$ which are defined for $t \in [0, 1]$. It is possible to check that $(\tilde{R}^{(n)}_t )_{0 \leq t \leq 1} \overset{d}{=} (R^{(n)}_t )_{0 \leq t \leq 1}$ and $(\tilde{X}^{(n)}_t )_{0 \leq t \leq 1} \overset{d}{=} (X^{(n)}_t )_{0 \leq t \leq 1}$ by checking that the characteristics of these processes are equal (since $[n] - [n(1-t)] - 1 = \mathrm{ceil}(nt) - 1 = [nt]$, where $\mathrm{ceil}$ is the ceiling function), and by applying Theorem II.4.25 p.110 in \cite{js2003}. (See also the example on p.97 in \cite{js2003} for the computation of the characteristics.) Thus, we can imitate the proof of Theorem 3.1. in \cite{carmonapetityor2001} to obtain 
\begin{equation*}
\begin{split}
\theta^{(n)}_1 & = e^{R^{(n)}_1}y + \int_{0+}^1e^{R^{(n)}_1 - R^{(n)}_{s-}}dX^{(n)}_s \overset{d}{=} e^{\tilde{R}^{(n)}_1}y + \int_{0+}^1e^{\tilde{R}^{(n)}_{u-}}d\tilde{X}^{(n)}_u \\
& \overset{d}{=} e^{R^{(n)}_1}y + \int_{0+}^1e^{R^{(n)}_{u-}}dX^{(n)}_u.
\end{split}
\end{equation*}

Then, using the fact that $|a + b|^{q} \leq 2^{q-1}(|a|^q + |b|^q)$, we obtain
$$\eE\left(|\theta^{(n)}_1|^q\right) \leq 2^{q-1}\left[y^q\eE\left(e^{qR^{(n)}_1}\right) + \eE\left(\left|\int_{0+}^1 e^{R^{(n)}_{u-}}dX^{(n)}_u\right|^q\right)\right].$$
Denote by $I_1^{(n)}$ and $I_2^{(n)}$ the expectation appearing on the r.h.s. of the above inequality. We will treat each expectation separately.

For $I^{(n)}_1$ we simply have
\begin{equation}\label{eq_finiteR}
\sup_{n \in \nN^*}I^{(n)}_1 = \sup_{n \in \nN^*}\prod_{i = 1}^{n}\eE(e^{q\gamma^{(n)}_i}) = \sup_{n \in \nN^*}\eE\left(e^{q\gamma^{(n)}_1}\right)^n < \infty.
\end{equation}

For $I^{(n)}_2$, we start by defining $M^{(n)}_t = \sum_{i = 1}^{[nt]} \frac{\ln(\rho_i) - \mu_{\rho}}{\sqrt{n}}$, for $0 \leq t \leq 1$. It is possible to check that $(M^{(n)}_t)_{0 \leq t \leq 1}$ is a martingale, for each $n \in \nN^*$ (for the filtration defined above Theorem \ref{thm_1}.) In fact, the martingale property is checked in the same manner as for $X^{(n)}$ in the proof of Theorem \ref{thm_1} and the integrability is clear. Thus, $(e^{M^{(n)}_t})_{0 \leq t \leq 1}$ is a submartingale, and using Doob's inequality we obtain
\begin{equation*}
\begin{split}
\eE\left(\left|\sup_{0 \leq t \leq 1}e^{R^{(n)}_t}\right|^q\right) & = \left(\sup_{0 \leq t \leq 1}e^{q\mu_{\rho}\frac{[nt]}{n}}\right)\eE\left(\left|\sup_{0 \leq t \leq 1}e^{M^{(n)}_t}\right|^q\right) \\
& \leq \mathrm{max}(1, e^{\mu_{\rho}q})\left(\frac{q}{q-1}\right)^q \eE(e^{qM^{(n)}_1})\\
& = \mathrm{max}(1, e^{-\mu_{\rho}q})\left(\frac{q}{q-1}\right)^q \eE(e^{qR^{(n)}_1}).
\end{split}
\end{equation*}
And so the assumption (\ref{asm_thm2}) together with (\ref{eq_finiteR}) imply that 
\begin{equation}\label{eq_supsupInteg}
\sup_{n \in \nN^*}\eE\left(\left|\sup_{0 \leq t \leq 1}e^{R^{(n)}_t}\right|^q\right) < \infty.
\end{equation}

Writing $X^{(n)}_t = \mu_{\xi}A^{(n)}_t + N^{(n)}_t$, with the processes defined in the proof of Theorem \ref{thm_1} and using the fact that $|A^{(n)}_t| \leq t$, for all $t \geq 0$, we have
\begin{equation*}
\begin{split}
I^{(n)}_2 & \leq 2^{q-1}\left[\eE\left(|\mu_{\xi}|\left|\int_{0+}^1e^{R^{(n)}_{s-}}dA^{(n)}_s\right|^q\right) + \eE\left(\left|\int_{0+}^1e^{R^{(n)}_{s-}}dN^{(n)}_s\right|^q\right)\right] \\
& \leq 2^{q-1}|\mu_{\xi}|^q\eE\left(\left|\sup_{0 \leq t \leq 1}e^{R^{(n)}_1}\right|^q\right) + 2^{q-1}\eE\left(\left|\int_{0+}^1e^{R^{(n)}_{s-}}dN^{(n)}_s\right|^q\right).
\end{split}
\end{equation*}
But, from the Burkholder-Davis-Gundy inequality applied twice and the independence of the sequences, we obtain
\begin{equation*}
\begin{split}
\eE\left(\left|\int_{0+}^1e^{R^{(n)}_{s-}}dN^{(n)}_s\right|^q\right) & \leq D_q\eE\left(\left|\int_{0+}^1e^{2R^{(n)}_{s-}}d[N^{(n)}]_s\right|^{q/2}\right) \\
& \leq D_q\eE\left(\sup_{0 \leq t \leq 1}e^{qR^{(n)}_1}\right) \eE([N^{(n)}]_1^{q/2}) \\
& \leq d_q D_q \eE\left(\left|\sup_{0 \leq t \leq 1}e^{R^{(n)}_1}\right|^q\right) \eE(|N^{(n)}_1|^q)
\end{split}
\end{equation*}
for some positive constants $D_q$ and $d_q$. Thus,
\begin{equation*}
\begin{split}
\sup_{n \in \nN^*}I^{(n)}_2 & \leq 2^{q-1}|\mu_{\xi}|\sup_{n \in \nN^*}\eE\left(\left|\sup_{0 \leq t \leq 1}e^{R^{(n)}_t}\right|^q\right) \\
& + 2^{q-1}d_q D_q\sup_{n \in \nN^*}\eE\left(\left|\sup_{0 \leq t \leq 1}e^{R^{(n)}_t}\right|^q\right)\sup_{n \in \nN^*}\eE\left(|N^{(n)}_1|^q\right),
\end{split}\
\end{equation*}
and so (\ref{eq_supsupInteg}) implies that it is enough to check that $\sup_{n \in \nN^*}\eE(|N^{(n)}_1|^q)$ is finite. (Note that a similar argument to the one used to prove the finiteness of $I^{(n)}_2$ is used in a different context in the proof of Lemma 5.1 in \cite{bankovsky}.)

Using the multinomial theorem, we obtain
\begin{equation*}
\begin{split}
\eE(|N^{(n)}_1|^q) = \sum_{k_1 + \dots + k_n = q} \binom{q}{k_1, \dots, k_n} \prod_{i = 1}^{n}\eE\left(\left(\frac{\xi_1 - \mu_{\xi}}{\sqrt{n}}\right)^{k_i}\right),
\end{split}
\end{equation*}
where the sum is taken over all non-negative integer solutions of $k_1 + \dots + k_n = q$. Since $\eE((\xi_1 - \mu_{\xi})^{k_i}) = 0$, when $k_i = 1$, we can sum over the partitions with $k_i \neq 1$ for all $i = 1, \dots, n$. Now, since $k_i/q \leq 1$, we have by Jensen's inequality
\begin{equation*}
\begin{split}
\prod_{i = 1}^{n}\eE\left(\left(\frac{\xi_1 - \mu_{\xi}}{\sqrt{n}}\right)^{k_i}\right) & \leq \prod_{i = 1}^{n}\eE\left(\left|\frac{\xi_1 - \mu_{\xi}}{\sqrt{n}}\right|^{q}\right)^{\frac{k_i}{q}} \\
& = n^{-q/2}\eE(|\xi_1 - \mu_{\xi}|^{q}).
\end{split}
\end{equation*}
Thus,
$$I^{(n)}_3 \leq 2^{q-1} \eE(|\xi_1 - \mu_{\xi}|^{q}) n^{-q/2} \sum_{k_1 + \dots + k_n = q, k_i \neq 1} \binom{q}{k_1, \dots, k_n}.$$
where $I_3^{(n)} = \eE(|N_1^{(n)}|^q)$ and since we need to take the supremum over $n \in \nN^*$, we need to check that the r.h.s. of the above inequality is bounded in $n$. For this, note that the sum of multinomial coefficients is equal to
\begin{equation}\label{eq_binomMom}
\sum_{i = 1}^{[q/2]}\binom{n}{i}\sum_{l_1 + \dots + l_i = q - 2i} \binom{q}{l_1 + 2, \dots, l_i + 2},
\end{equation}
where, for each $i = 1, \dots, [q/2]$, the second sum is taken over all non-negative integer solutions of $l_1 + \dots + l_i = q - 2i$. This follows from the fact that if $(k_1, \dots, k_n)$ are non-negative integer solutions of $k_1 + \dots + k_n = q$, we have, since $k_i \neq 1$, that the number of non-zero terms in $(k_1, \dots, k_n)$ is at most $[q/2]$. Thus, letting $i$ be the number of non-zero terms and $(j_1, \dots, j_i)$ their indices, we find that $(k_{j_1} - 2) + \dots + (k_{j_i} - 2) = q - 2i$, which yields the claimed equality. Then, we have
$$\binom{q}{l_1 + 2, \dots, l_i + 2} \leq C_i \binom{q - 2i}{l_1, \dots, l_i},$$
with $C_i = 2^{-i}q!/(q-2i)!$ and
$$\sum_{l_1 + \dots + l_i = q - 2i}\binom{q-2i}{l_1, \dots, l_i} = i^{q-2i}.$$
Let $C_q = \max_{i = 1, \dots, [q/2]}{C_i}$ and $K_q = 2^{q-1} C_q \eE(|\xi_1 - \mu_{\xi}|^{q})$, remark that the binomial coefficient in (\ref{eq_binomMom}) is bounded by $n^{[q/2]}$ and that
$$I^{(n)}_3 \leq K_qn^{-q/2}\sum_{i = 1}^{[q/2]}\binom{n}{i}i^{q-2i} \leq K_q n^{-q/2 + [q/2]}\sum_{i = 1}^{[q/2]}i^{q-2i},$$
which is bounded in $n$. Thus, $\sup_{n \in \nN^*} I^{(n)}_3 < \infty$, $\sup_{n \in \nN^*}I^{(n)}_2 < \infty$ and
$$\sup_{n \in \nN^*}\eE\left(|\theta^{(n)}_1|^q\right) \leq 2^{q-1}y^q\sup_{n \in \nN^*} I^{(n)}_1 + 2^{q-1}\sup_{n \in \nN^*} I^{(n)}_2 < \infty.$$
So, the sequence $((\theta^{(n)}_1)^p)_{n \in \nN^*}$ is uniformly integrable and we have $\lim_{n \to \infty} \eE[(\theta^{(n)}_1)^p] = \eE[(Y_1)^p]$.
\end{proof}

\section*{Acknowledgements}

The authors thank Lioudmila Vostrikova for helpful discussions and comments on the paper. They also would like to acknowledge financial support from the D\'{e}fiMaths project of the "F\'{e}d\'{e}ration de Recherche Math\'{e}matique des Pays de Loire" and from the PANORisk project of the "R\'{e}gion Pays de la Loire".

\bibliographystyle{plain}
\bibliography{func_conv}

\end{document}